\documentclass[11pt]{amsart}
\usepackage{amsmath,amssymb,amsthm}
\usepackage[margin=1.25in]{geometry}

\newtheorem{theorem}{Theorem}
\newtheorem{lemma}[theorem]{Lemma}
\newtheorem{corollary}[theorem]{Corollary}
\theoremstyle{remark}
\newtheorem{remark}[theorem]{Remark}

\DeclareMathOperator{\lgg}{lg}

\title{Sums of distinct divisors of factorials}
\author{Scott D. Hughes}
\address{INDEPENDENT RESEARCHER}
\email{v@deltagray.com}
\date{}
\dedicatory{\normalfont September 9, 2026}

\begin{document}

\begin{abstract}
For practical $N$ let $h(N)$ be the least $k$ such that every integer
$1\le m\le N$ is a sum of at most $k$ distinct divisors of $N$.  We
prove
\[
  h(n!)\;\le\;(2\log2+o(1))\,\frac n{\log n}.
\]
This improves the bounds of order $n/(\log n)^{1/2-\varepsilon}$
established in Tenenbaum--Yokota's Lemma~4 and Yokota's 1995 knapsack
note.  We combine their decreasing greedy construction with the sharper
factorial divisor-gap estimate of Berend--Harmse.  Counting the steps
separately below and above $\sqrt{n!}$, with the upper range handled
through reciprocal divisors, retains the leading coefficient in the
gap exponent and yields the explicit constant $2\log2$.
\end{abstract}

\maketitle

Throughout, $\log$ denotes the natural logarithm and $\lgg t=\log_2t$.

\section{Statement}

An integer $N\ge1$ is \emph{practical} if every $1\le m\le N$ is a sum
of distinct divisors of $N$; for practical $N$ let
\[
  h(N)=\min\bigl\{k:\ \text{every }1\le m\le N\text{ is a sum of at
  most }k\text{ distinct divisors of }N\bigr\}.
\]
Erd\H{o}s proved $h(n!)\le n$ and asked for the true order of
$h(n!)$ \cite[pp.~37--38]{ErGr80}.  Lemma~4 of Tenenbaum--Yokota \cite{TY90} is
stated with $N_k$ denoting either $k!$ or $\prod_{j\le k}p_j$; in
particular, for every $\eta>0$ it gives
\[
  h(n!)\;\ll_\eta\;\frac n{(\log n)^{1/2-\eta}}.
\]
Yokota's 1995 note \cite{Yo95} reaches the same exponent by a
decreasing greedy construction.  For every fixed
$0<\varepsilon<\tfrac12$ and all sufficiently large $k$, its proof
represents each integer $1\le S<2k!$ as a sum of distinct divisors of
$k!$, using
\[
  m\;\ll_\varepsilon\;\frac k{(\log k)^{1/2-\varepsilon}}
\]
terms.  Its Lemma~3 uses the older gap input
$\varepsilon_k=\exp\{-(\log k)^{3/2-\varepsilon}\}$, furnished by
Tenenbaum's factorial small-interval estimates; see the deduction on
pp.~6--7 and Theorem~4 of \cite{Te87}, with the corrections in
\cite{Te00}.  Neither \cite{TY90} nor \cite{Yo95} states a bound of
order $n/\log n$.  Theorem~\ref{thm:main} gives this order with an
explicit leading constant.

\begin{theorem}\label{thm:main}
$\displaystyle h(n!)\;\le\;(2\log2+o(1))\,\frac n{\log n}$ as
$n\to\infty$.
\end{theorem}

The proof retains the greedy construction of \cite{TY90,Yo95} and
uses Berend--Harmse's 1993 factorial gap estimate directly.  We
recompute the step counts in both ranges, retaining the coefficient
$1/(2\log2)$ in the gap exponent to obtain the constant $2\log2$.
Sections 2 and 3 give the details.

\section{The two ingredients}

\begin{theorem}[{Berend--Harmse \cite[Thm.~2]{BH93}}]\label{thm:BH}
For $n\ge2^{16}$ and $\sqrt{(n-1)!}\le D\le\sqrt{n!}$, there is a
divisor $x$ of $n!$ such that
\[
  \Bigl|\frac xD-1\Bigr|
  \;\le\;5\cdot10^{7}
  \Bigl(\frac{\lgg n}{n}\Bigr)^{\frac{\lgg n-\lgg(\lgg n)+1}2+\lgg e}
  \;\le\;\Bigl(\frac1n\Bigr)^{\frac{\lgg n}2-\lgg(\lgg n)}.
\]
\end{theorem}

Put
\[
  \varepsilon_j
  \;:=\;\Bigl(\frac1j\Bigr)^{\frac{\lgg j}2-\lgg(\lgg j)},
\]
the rightmost bound above at $n=j$, so that
\begin{equation}\label{eq:eps}
  \log\frac1{\varepsilon_j}
  \;=\;\Bigl(\frac{\lgg j}2-\lgg(\lgg j)\Bigr)\log j
  \;=\;\frac{(\log j)^2}{2\log2}
  \Bigl(1-\frac{2\log(\log j/\log2)}{\log j}\Bigr).
\end{equation}
Since $\bigl(\tfrac{\lgg x}2-\lgg(\lgg x)\bigr)\log x$ is increasing
for $x\ge2^{16}$, the sequence $(\varepsilon_j)_{j\ge2^{16}}$ is
decreasing.

\begin{corollary}\label{cor:gap}
Let $j\ge2^{16}$, and let $a<b$ be consecutive divisors of $n!$
(any $n\ge j$) with $\sqrt{(j-1)!}\le\sqrt{ab}\le\sqrt{j!}$.  Then
$\log(b/a)\le3\varepsilon_j$.
\end{corollary}

\begin{proof}
Apply Theorem~\ref{thm:BH} at $D=\sqrt{ab}$; the resulting divisor $x$
of $j!$ divides $n!$.  Since $a,b$ are consecutive divisors of $n!$,
either $x\le a$ or $x\ge b$.  If $x\le a$, then
$1-\varepsilon_j\le x/D\le a/D=\sqrt{a/b}$, so
$b/a\le(1-\varepsilon_j)^{-2}$; if $x\ge b$, then
$\sqrt{b/a}=b/D\le x/D\le1+\varepsilon_j$, so
$b/a\le(1+\varepsilon_j)^2$.  Since
$-\log(1-u)\le u/(1-u)\le\tfrac43u$ for $0\le u\le\tfrac14$, both
cases give $\log(b/a)\le3\varepsilon_j$.  Here
$\varepsilon_j\le\varepsilon_{2^{16}}=(2^{-16})^{8-4}=2^{-64}<\tfrac14$.
\end{proof}

The second ingredient is the greedy mechanism used in the proof of
\cite[Lemma~4]{TY90}.  We use the standard fact, recalled there, that
the ratio of two consecutive divisors of $n!$ does not exceed $2$.

\begin{lemma}[greedy step]\label{lem:greedy}
Let $N$ be an integer whose consecutive divisors have ratio at most
$2$, and let $1\le R\le N$.  If $R$ divides $N$ (in particular if
$R=N$), the greedy expansion terminates at this step.  Otherwise, let
$d<R<b$ be the consecutive divisors of $N$ bracketing $R$.  Then
\[
  R-d\;<\;d,
  \qquad
  R-d\;\le\;2R\log\frac bd .
\]
Consequently successive divisors chosen by the greedy expansion are
strictly decreasing, hence distinct.
\end{lemma}

\begin{proof}
Since $R<b\le2d$, we have $R-d<b-d\le d$, so the next chosen divisor
is smaller than $d$.  Moreover
$R-d<b-d=d(b/d-1)\le R\,(b/d-1)\le2R\log(b/d)$, because
$b/d\in[1,2]$ and $x-1\le2\log x$ on that interval.
\end{proof}

\section{Proof of Theorem \ref{thm:main}}

Set $j_0=2^{16}$ and $T_0=2\sqrt{(j_0+1)!}$, an absolute constant, and
write $N=n!$.  We may assume $n\ge j_0^2$, since the theorem is
asymptotic.  Let $1\le m\le N$ and run the greedy expansion; write
$R_0=m>R_1>\cdots$ for the successive remainders (the expansion
terminates whenever a remainder divides $N$) and $\ell(R)=\log R$.
All step counts below concern nonterminal subtractions; if the
process terminates at a positive divisor of $N$, that final divisor
contributes one additional term, absorbed in the $O(1)$ endgame
count.

Two symmetric facts about the bracketing divisors $d<R<b$ locate the
Berend--Harmse window.  If $R\le\sqrt N$ then $db\le N$: otherwise
$N/d<b$, while $N/d\ge N/R\ge\sqrt N\ge R>d$, so $N/d$ would be a
divisor strictly between $d$ and $b$.  Symmetrically, if $R>\sqrt N$
then $db\ge N$.

\subsection*{Lower range $T_0\le R\le\sqrt N$}
Let $j(R)\in[j_0+1,n]$ be minimal with $R\le\sqrt{j(R)!}$.  For a
nonterminal step at $R$ with bracketing divisors $d<R<b$, the
geometric mean $\sqrt{db}$ lies in $[R/2,2R]$ (as $d>R/2$, $b<2R$)
and, by the preceding paragraph, $\sqrt{db}\le\sqrt N$.  Each window
$[\sqrt{(j-1)!},\sqrt{j!}]$ has logarithmic width
$\tfrac12\log j\ge\log2$, so $\sqrt{db}$ lies in the window of index
$j(R)-1$, $j(R)$, or $j(R)+1$; moreover $j(R)-1\ge j_0$ by the choice
of $T_0$.  When $j(R)=n$, the shared endpoint $\sqrt N$ is assigned to
the window of index $n$.  Corollary \ref{cor:gap} and the monotonicity
of $\varepsilon_j$ give
$\log(b/d)\le3\varepsilon_{j(R)-1}=:\tfrac12\delta_{j(R)}$, and by
\eqref{eq:eps}, applied at $j-1$ (the index shift and the constant
$\log6$ are absorbed in the error term),
\[
  s_j\;:=\;\log\frac1{\delta_j}
  \;=\;\frac{(\log j)^2}{2\log2}
  \Bigl(1+O\Bigl(\frac{\log\log j}{\log j}\Bigr)\Bigr).
\]
Lemma \ref{lem:greedy} gives
$R_{i+1}\le2R_i\log(b/d)\le R_i\,\delta_{j(R_i)}$, i.e.
\[
  \ell(R_{i+1})\;\le\;\ell(R_i)-s_{j(R_i)} .
\]
Each step descends by at least $s_{j(R_i)}$; moreover $s_j$ is
increasing in $j$ and $j(R)$ is nondecreasing in $R$, so
$s_{j(e^\ell)}\le s_{j(R_i)}$ throughout the $i$-th step's interval
$[\ell(R_{i+1}),\ell(R_i)]$, whence
$\int_{\ell(R_{i+1})}^{\ell(R_i)}d\ell/s_{j(e^\ell)}\ge1$.  The step
intervals are disjoint, and all but possibly the final one (which may
cross $\log T_0$) lie inside $[\log T_0,\tfrac12\log N]$.  Hence
\begin{align*}
  \#\{\text{steps with }T_0\le R\le\sqrt N\}
  &\le 1+\int_{\log T_0}^{\frac12\log N}\frac{d\ell}{s_{j(e^\ell)}}
  \;\le\;O(1)+\sum_{j=j_0+1}^{n}\frac{\tfrac12\log j}{s_j}\\
  &=\log2\sum_{j=j_0+1}^{n}\frac1{\log j}
   +O\Bigl(\sum_{j=j_0+1}^{n}\frac{\log\log j}{(\log j)^2}\Bigr)+O(1)\\
  &=\Bigl(\log2+O\Bigl(\frac{\log\log n}{\log n}\Bigr)\Bigr)
   \frac n{\log n},
\end{align*}
since the integral over the window of index $j$, of logarithmic
width $\tfrac12\log j$, contributes at most $\tfrac12\log j/s_j$,
and $\sum_{j\le n}1/\log j\sim n/\log n$.

\subsection*{Upper range $\sqrt N<R\le N/T_0$}
Put $U=N/R$.  If $d<R<b$ are the bracketing divisors of $R$, then
$N/b<U<N/d$ are consecutive divisors of $N$ with the same ratio
$b/d$; their geometric mean $N/\sqrt{db}$ lies in $[U/2,2U]$ and, by
$db\ge N$, is at most $\sqrt N$.  Thus the window argument of the
lower range applies verbatim \emph{at the mirror position}, giving
\[
  R_{i+1}\;\le\;\delta_{j(U_i)}\,R_i,
  \qquad\text{i.e.}\qquad
  \log U_{i+1}-\log U_i\;\ge\;s_{j(U_i)},
\]
where now $U_i=N/R_i$ \emph{increases} along the expansion.  The
charging integral of the lower range cannot simply be mirrored — the
monotonicity of $s_{j(U)}$ runs the wrong way — so we count these
steps dyadically instead.

Let
\[
  R_n:=\left\lfloor\log_2\frac{n}{2j_0}\right\rfloor,
  \qquad J_r:=\frac{n}{2^{r+1}}
  \quad(0\le r\le R_n),
\]
and let
\[
  \mathcal B_r:=\{j\in\mathbb Z:J_r<j\le2J_r\}.
\]
Then $J_r\ge j_0$ for $0\le r\le R_n$, and these blocks cover all
indices $j>J_{R_n}$ up to $n$; the remaining indices satisfy
$j\le J_{R_n}<2j_0$.  Consider the steps whose start satisfies
$j(U_i)\in\mathcal B_r$.  Their
$\log U_i$ lie in an interval of length at most
\[
  L_r\;\le\;\tfrac12\sum_{J_r<j\le2J_r}\log j
  \;=\;\tfrac12J_r\log J_r+O(J_r),
\]
where the sum is over integers.  Successive starts in this block are
separated in $\log U$ by at least $s_{\lfloor J_r\rfloor+1}$, so their
number is at most
\[
  1+\frac{L_r}{s_{\lfloor J_r\rfloor+1}}
  \;=\;\Bigl(\log2+O\Bigl(\frac{\log\log J_r}{\log J_r}\Bigr)\Bigr)
  \frac{J_r}{\log J_r}+O(1) .
\]
The relative-error terms are summable at the required scale:
\[
  \sum_{r=0}^{R_n}
    \frac{J_r\log\log J_r}{(\log J_r)^2}
  =O\!\left(\frac{n\log\log n}{(\log n)^2}\right)
  =o\!\left(\frac n{\log n}\right).
\]
Indeed, put $g(t)=\log\log t/(\log t)^2$.  The function $g$ is
decreasing for $t\ge j_0$.  For the blocks with $J_r\ge\sqrt n$,
$g(J_r)\le g(\sqrt n)$ and $\sum J_r=O(n)$, giving
$O(n\log\log n/(\log n)^2)$.  For the remaining blocks,
$g(J_r)=O(1)$ and $\sum J_r=O(\sqrt n)$, which is
$o(n/\log n)$.
The remaining indices have $j(U_i)<2j_0$, hence
$T_0\le U_i\le\sqrt{(2j_0)!}$; this is a fixed logarithmic interval,
and the positive lower bound for the step size gives $O(1)$ further
steps.  The accumulated $O(1)$ terms from the $O(\log n)$ dyadic
blocks total $O(\log n)=o(n/\log n)$.  Summing over the blocks with
$J_r\ge j_0$, and using
$\sum_{r=0}^{R_n}J_r/\log J_r=(1+o(1))\,n/\log n$,
\[
  \#\{\text{steps with }\sqrt N<R\le N/T_0\}
  \;\le\;(\log2+o(1))\,\frac n{\log n}.
\]

\subsection*{Endgames}
For $N/T_0<R\le N$ the mirror $U=N/R$ lies below $T_0$; every
nonterminal step still has $d>R/2$, hence $R_{i+1}<R_i/2$, so this
region contributes at most $\log T_0/\log2+O(1)=O(1)$ steps.  The same
halving argument disposes of $1\le R<T_0$ at the bottom.  Adding the
two main passages and the $O(1)$ endgames proves the theorem. \qed

\section{Remarks}

\begin{remark}
The proof gives the more explicit relative error
$O(\log\log n/\log n)$ in the constant, inherited from the
$\lgg(\lgg n)$ term in Theorem \ref{thm:BH}; Berend--Harmse note
(their Remark~2) that up to a bounded power of $\lgg n$ their bound is
the best their method produces.
\end{remark}

\begin{remark}
In the other direction, writing $T=\tau(n!)$ and $k=h(n!)$: if
$k\le T/2$ then
\[
  n!\;\le\;\sum_{i=0}^{k}\binom Ti\;\le\;(eT/k)^k .
\]
We only need the estimate $\log T\ll n/\log n$, which follows from
Chebyshev's bound $\pi(x)\ll x/\log x$.  Indeed, in
$\log T=\sum_{p\le n}\log(v_p(n!)+1)$, where $v_p$ denotes the
$p$-adic valuation, primes $p\le\sqrt n$ contribute
$O(\sqrt n\log n)$.  For $p>\sqrt n$ in
$(n/2^{r+1},n/2^r]$, we have $v_p(n!)=\lfloor n/p\rfloor<2^{r+1}$,
and there are $O(n/(2^r\log n))$ such primes.  Their total contribution
is therefore
\[
  \ll\frac n{\log n}\sum_{r\ge0}\frac{r+1}{2^r}
  \ll\frac n{\log n}.
\]
Taking logarithms in the counting inequality and using
$\log(n!)\asymp n\log n$ now gives $k\gg(\log n)^2$.
If $k>T/2$, the same conclusion follows from $T\ge n$, since every
integer from $1$ to $n$ divides $n!$.  The gap between $(\log n)^2$
and $n/\log n$ remains; Erd\H{o}s asked whether $h(n!)<n^{o(1)}$, or
even $h(n!)<(\log n)^{O(1)}$ \cite[pp.~37--38]{ErGr80}.
\end{remark}

\begin{remark}
Yokota's note \cite{Yo95} uses the same two greedy ranges separated
at $\sqrt{k!}$, with the upper range mirrored through $k!/S$.  Its
Lemma~2 supplies the ratio-at-most-$2$ fact used in
Lemma~\ref{lem:greedy}.  The sharper estimate of Berend--Harmse
\cite[Thm.~2]{BH93} gives
\[
  \log(1/\varepsilon_j)=\frac{(\log j)^2}{2\log2}
    +O(\log j\log\log j).
\]
The present proof applies this estimate directly and counts the steps
in both ranges to retain its leading coefficient.  The contribution
of Theorem~\ref{thm:main} is the resulting bound with constant
$2\log2$; the greedy construction is due to Tenenbaum--Yokota and
Yokota, and the sharper gap estimate is due to Berend--Harmse.
Berend--Harmse appeared in 1993; Yokota's paper was received in
September 1994 and cites Tenenbaum's 1987 paper and Hardy--Wright.
\end{remark}

\end{document}